\documentclass[11pt]{amsart}
\usepackage{amsthm}
\usepackage{amsmath}
\usepackage{amsfonts}
\usepackage{amssymb}
\usepackage{url}
\usepackage{epsfig}

\usepackage{graphicx,amssymb,amsmath,amsthm}
\newtheorem{theorem}{Theorem}[section]
\newtheorem{lemma}[theorem]{Lemma}

\theoremstyle{definition}

\newtheorem{example}[theorem]{Example}

\newcommand{\lra}{\longrightarrow}
\theoremstyle{remark}

\newcommand{\A}{\mathcal{A}}
\newcommand{\x}{\tilde{x}}
\newcommand{\R}{\mathbb{R}}
\newcommand{\E}{\mathcal{E}}

\newcommand{\T}{\mathbb{T}}
\newcommand{\Z}{\mathbb{Z}}
\newcommand{\N}{\mathbb{N}}

\newcommand{\MS}{\mathbb{S}}
\newcommand{\F}{\mathcal F}
\renewcommand{\H}{\mathcal H}

\newcommand{\EE}{\mathbb E}
\begin{document}

\title{The Performance of PCM Quantization Under Tight Frame Representations}

\author{Yang Wang}
\thanks{Yang Wang was supported in part by the
       National Science Foundation grant DMS-0813750, DMS-08135022 and DMS-1043034. Zhiqiang Xu was supported  by NSFC grant 10871196 and by the Funds for Creative Research Groups of China (Grant No. 11021101).}
\address{Department of Mathematics  \\ Michigan State University\\
East Lansing, MI 48824, USA}
\email{ywang@math.msu.edu}

\author{Zhiqiang Xu}
\address{LSEC, Inst.~Comp.~Math., Academy of
Mathematics and System Sciences,  Chinese Academy of Sciences, Beijing, 100091, China}
\email{xuzq@lsec.cc.ac.cn}

\maketitle
\begin{abstract}
In this paper, we study the performance of the PCM scheme with
linear quantization rule for quantizing finite unit-norm tight
frame expansions for $\R^d$ and derive the PCM quantization error
{\em without} the White Noise Hypothesis. We prove that for the class of unit norm tight frames derived from uniform frame paths the quantization error has
an upper bound of $O(\delta^{3/2})$ regardless of the frame redundancy. This is achieved
using some of the techniques developed by G\"{u}nt\"{u}rk in his study of
Sigma-Delta quantization. Using tools of harmonic analysis we show that
this upper bound is sharp for $d=2$. A consequence of this result is that,
unlike with Sigma-Delta quantization, the error for PCM quantization in general
does not diminish to zero as one increases the frame redundancy.
We extend the result to high dimension and show that the
PCM quantization error has an upper bound  $O(\delta^{(d+1)/2})$ for asymptopitcally
equidistributed unit-norm tight frame of $\R^{d}$.
\end{abstract}

\section{Introduction}

In signal processing, coding and many other practical applications it is important to find a
suitable representation for a given signal. In general, the first step towards this objective is
finding an atomic decomposition of the signal using a given set of {\em atoms}, or {\em
dictionary}. In this approach, we assume that the signal $x$ is an element of a finite-dimensional
Hilbert space $H=\R^d$ and $x$ is represented as a linear combination of $\{e_j\}_{j=1}^N$, i.e.,
\begin{equation}\label{eq:expan}
x\,\,=\,\, \sum_{j=1}^N c_j e_j,
\end{equation}
where $c_j$ are real numbers. In practical application, instead of a true basis, $\{e_j\}_{j=1}^N$
is chosen to be a {\em frame}. Given ${\mathcal F}=\{e_j\}_{j=1}^N$, we let $F=[e_1,\ldots,e_N]$ be
the corresponding matrix whose columns are $\{e_j\}_{j=1}^N$. We say ${\mathcal F}$ is a {\em
frame} of $\R^d$ if the matrix $F$ has rank $d$. The frame $\F$ is {\em tight} with frame constant
$\lambda$ if $FF^*=\lambda I_d$.  The matrix $F^T$ as an operator $F^*:\R^d\rightarrow \R^N$ is
often known as the {\em analysis operator} with respect to the frame $\F$, where
$(F^*x)_j=\left<x,e_j\right>$. The adjoint operator given by $F: \R^N\rightarrow \R^d$,
$Fy=\sum_{j=1}^Ny_je_j$ is known as the {\em synthesis operator} with respect to $\F$. We call the
operator $S:=FF^*$ as the {\em frame operator}. Then $\{S^{-1}e_j\}_{j=1}^N$ is called the {\em
canonical dual frame} of the frame $\F$. It is easy to see that for any $x\in\R^d$ we have the
reconstruction formula
\begin{equation}\label{eq:frame}
     x=\sum_{j=1}^N\left<x,e_j\right>(S^{-1}e_j)=\sum_{j=1}^N\left<x,S^{-1}e_j\right>e_j.
\end{equation}
If $\F$ is a tight frame with frame bound $\lambda$ then clearly $S^{-1}e_j=\lambda^{-1}e_j$. In
particular, for the important case of {\em finite unit-norm tight frames} in which $\|e_j\|=1$ for
all $j$ we have $\lambda=N/d$ and (\ref{eq:frame}) is reduced to
$$
x=\frac{d}{N}\sum_{j=1}^N\left<x,e_j\right>e_j \quad\quad\mbox{for all } x\in \R^d.
$$

In the digital domain the representation must be quantized. In other words, the coefficients $\left<x,e_j\right>$ from the analysis operator  must be mapped to a discrete set of values $\A$ called the {\em
quantization alphabet}. The simplest way  for such a mapping is the {\em Pulse Code Modulation
(PCM)} quantization scheme, which has $\A=\delta\Z$ with $\delta>0$ and maps a value $t$ the value
in $\A$ that is the closest to $t$. More precisely, the mapping is done by the function
$$
Q_\delta(t)\,\,:=\,\, {\rm arg min}_{r\in \A} |t-r|\,\,=\,\, \delta \left\lfloor
\frac{t}{\delta}+\frac{1}{2}\right\rfloor
$$
and the quantization function $Q_\delta$ is called the {\em quantizer}. Thus in practical applications
we in fact have only a quantized representation through the quantized analysis operator
$\tilde y_j:=Q_\delta((F^*x)_j)=Q_\delta(\left<x,e_j\right>),j=1,\ldots,N$ for each $x\in\R^d$. The reconstruction
through the frame operator yields
$$
   \tilde x\,\,=\,\,\sum_{j=1}^NQ_\delta(\left<x,e_j\right>)(S^{-1}e_j) \quad\quad\mbox{for all } x\in \R^d.
$$
Naturally we may want to ask about the error for this reconstruction.

An important class of frames is the unit-norm tight frames. This paper shall focus on this class of
frames, although the questions we raise and the techniques we use can be applied to other frames.
Let $\F=\{e_j\}_{j=1}^N$ be a unit-norm tight frame in $\R^d$. For each $x\in\R^d$  we have
\begin{equation}\label{eq:expansion}
    x=\frac{d}{N}\sum_{j=1}^Nc_je_j, \quad\quad \mbox{where}\quad \quad  c_j=\left<x,e_j\right>.
\end{equation}
With PCM quantization and quantization alphabet $\A=\delta\Z$ the reconstruction becomes
\begin{equation}\label{eq:quanrecon}
   \x_\F\,\,=\,\,\frac{d}{N}\sum_{j=1}^Nq_j e_j, \quad\quad {\rm where }\quad \quad
q_j=Q_\delta(c_j)\in \A.
\end{equation}
Under this quantization we denote the reconstruction error by
$$
E_\delta(x,\F)\,\,:=\,\, \|x-\x_\F\|
$$
where $\|\cdot \|$ is $\ell_2$ norm. An important question is how $E_\delta(x,\F)$ behaves for a
given frame $\F$ and either for a given $x$ or for a given distribution of $x$. To simplify the
problem, the so-called {\em White Noise Hypothesis } (WNH) is employed by engineers and
mathematicians in this area (see \cite{GoVeTh98,Benn48,BePoYi06,BePoYi06II,BorWan09, JiWaWa07}).
 The WNH asserts that the quantization error sequence $\{x_j-q_j\}_{j=1}^N$ can be modeled as
 an independent sequence of i.i.d. random variables that are uniformly distributed on the
  interval $(-\delta/2, \delta/2)$. With the WNH, one can obtain the mean square error
$$
     MSE\,\,=\,\,{\mathcal E}(\|x-\x_\F\|^2)\,\,=\,\,\frac{d^2\delta^2}{12 N}.
$$
It has been shown that the WNH is asymptotically correct for fine quantizations (i.e. as $\delta$
tends to 0) under rather general conditions, see \cite{BorWan09, JiWaWa07}. Although the result implies
that the MSE decreases on the order of $1/N$, this is in fact quite misleading because the WNH
holds only asymptotically when the frame $\F$ (and hence $N$) is fixed while $\delta$ decreases to
0, and with a fixed $\delta$ WNH cannot hold whenever $N>d$ \cite{BorWan09}. Furthermore, the MSE only gives
 information about the average behavior of
quantization errors. There has not been an in-depth study on the behavior of the error
$E_\delta(x,\F)$ for a given $x$ and as one fixes $\delta$. This contrast sharply with the study
on the quantization error from the
Sigma-Delta quantization schemes, where the quantization step $\delta$ is typically assumed to be
fixed and rather coarse, see e.g. \cite{BePoYi06,BePoYi06II}. One of the objectives of this paper is to study
the behavior of $E_\delta(x,\F)$ as we choose different unit-norm tight frames $\F$.

It is well known that with the Sigma-Delta quantization schemes the reconstruction error will
diminish to 0 as we increase the redundancy of the frame $\F$. For unit-norm tight frames it means
that for fixed $\delta$ by letting $N \rightarrow \infty$ the reconstruction error tends to 0, even
when the quantization is coarse in the sense that $\delta>> 0$. One naturally asks whether similar
phenomenon also occurs with PCM quantizations, i.e. how much can we mitigate the reconstruction
error $E_\delta(x,\F)$ if we increase the redundancy of the frame $\F$, and is it possible that by
increasing redundancy in a suitable way the reconstruction error
$E_\delta(x,\F)$ be made arbitrarily small for all $x$? Clearly if $\|x\|< \delta/2$
then all coefficients are quantized to 0 and hence the quantization error is always $x$. Hence
in this case the error does not diminish. However, when $\|x\|>> \delta$ we may expect that
the increase redundancy will help mitigating the error.

In this paper we attempt a more in-depth study of the PCM quantization error $E_\delta(x,\F)$
with respect to unit-norm tight frames $\F$. In particular we study the asymptotic behavior
of $E_\delta(x,\F)$ as we increase the redundancy of the unit-norm tight frame. A surprising
result (at least to us)
is that {\em in general} the quantization error $E_\delta(x,\F)$ does not diminish to 0
no matter how much one increases the redundancy of the frame $\F$. The following example is
a good illustration.

\begin{example}\label{ex:num}
We choose $x_0=[\pi, e]^T$, $\delta=1/16$. Let $\tilde{\F}=\{e_j\}_{j=1}^N$ be the harmonic frame
in $\R^2$ given by $e_j=[\cos(\frac{2j\pi}{N}),\sin(\frac{2j\pi}{N})]^T$,
$j=0,\ldots,N-1$, which is a 2-dimensional unit-norm tight frame. Then we
compute $E_\delta(x_0,\tilde{\F})$ for $N=10,\ldots, 2000$ and show the
result in Figure 1. From the figure one can see that although as we increase $N$
the quantization error $E_\delta(x_0,\tilde{\F})$ decreases initially, it settles down to around positive value no matter how much redundancy is increased. Thus
PCM quantization fails to take advantage of redundance.
\begin{figure}[!ht]
\begin{center}
%\vspace{2in}
\epsfxsize=12cm\epsfbox{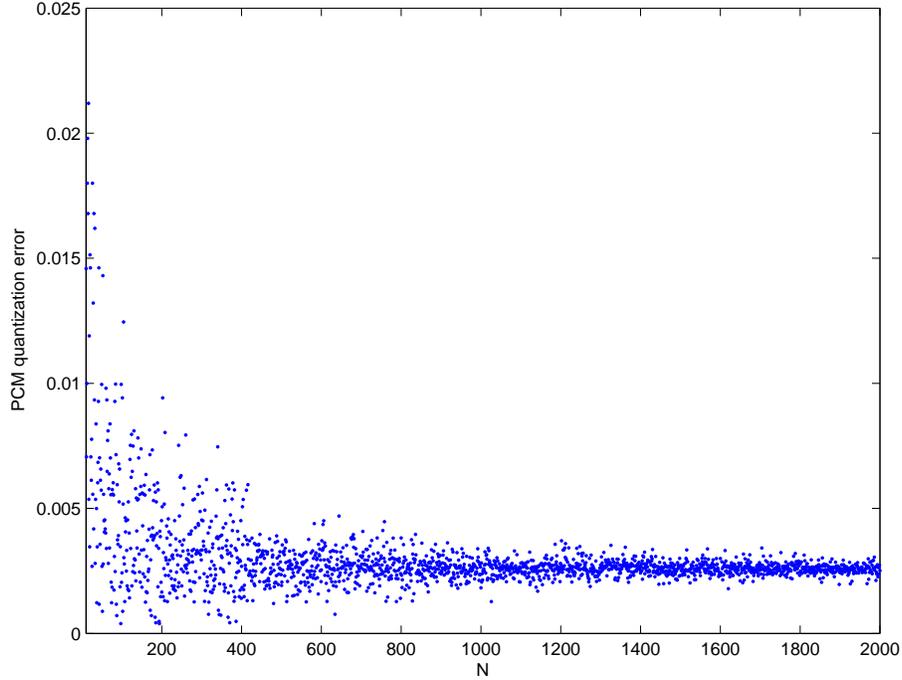}
\end{center}
\caption{ The frame expansion of $[\pi, e]^T\in \R^2$ with respect
to the frame $[\cos(2\pi j/N),\sin(2\pi j/N)]^T$ are quantized
using the PCM scheme with $\delta=1/16$. The figure shows the PCM
error  against the frame size $N$. }
\end{figure}
\end{example}

Of particular interest to this study is a very popular class of unit-norm tight frames
known as the {\em harmonic frames}. For any $N\geq d$ the harmonic frame
$\H_N^d=\{h_j^N\}_{j=0}^{N-1}$ is given by
$$
    h_j^N=\sqrt{\frac{2}{d}}\left[\cos\frac{2\pi j}{N},\sin \frac{2\pi
j}{N},\cos \frac{2\pi 2j}{N},\sin \frac{2\pi 2 j}{N},\ldots, \cos
\frac{2\pi \tilde d j  }{N},\sin \frac{2\pi \tilde d j }{N}\right]^T,
$$
if $d=2\tilde d$ is even or
$$
   h_j^N=\sqrt{\frac{2}{d}}\left[\frac{1}{\sqrt{2}},\cos\frac{2\pi
   j}{N},\sin \frac{2\pi j}{N},\ldots, \cos \frac{2\pi \tilde d j}{N},
    \sin \frac{2\pi \tilde d j }{N} \right]^T,
$$
if $d=2\tilde d+1$ is odd. Harmonic frames themselves are a special case of unit-norm tight frames
obtained from uniform frame paths introduced in \cite{BodPau07}. Let $f: [0,1] \lra \R^d$
be a continuous function with $\|f(t)\|=1$ for all $t$. It is called a
{\em uniform frame path} if for any $N \geq d$ the set of vectors
$\{f(\frac{j-1}{N})\}_{j=1}^{N}$ is a unit-norm tight frame in $\R^d$. So the
harmonic frame $\H_N^d$ is obtained simply by taking $N$ samples of the frame path
$$
    h(t) = \sqrt{\frac{2}{d}}\Bigl[\cos(t), \sin(t), \cos(2t), \sin(2t), \dots,
        \cos(\tilde dt), \sin(\tilde dt)\Bigr]^T
$$
if $d= 2\tilde d$ is even or
$$
    h(t) = \sqrt{\frac{2}{d}}\Bigl[\frac{1}{\sqrt 2}, \cos(t), \sin(t), \cos(2t), \sin(2t),
        \dots,\cos(\tilde dt), \sin(\tilde dt)\Bigr]^T
$$
if $d= 2\tilde d+1$ is odd. We examine the limitations of uniform frame paths in terms of
its ability to mitigate quantization errors with increasing redundancies.

Throughout the paper we shall use the notation $X\ll_{a,b,\ldots} Y$ to refer to the
inequality $X\leq C\cdot Y$, where the constant $C$ may depend on $a,b,\ldots,$ but
no other variable. We now state one of our main results in this paper.

\begin{theorem}\label{th:mupbound}
    Let $f: [0,1]\lra \R^d$ be a uniform frame path with bounded $f'$ and set
$\F:= \{f(\frac{j-1}{N})\}_{j=1}^N$. Suppose that $x\in\R^d$ and $h(t):=\left<x,f(t)\right>$
is analytic such that all zeros of $h''(t)$ on $[0,1]$ are simple. Then
$$
     E_\delta(x,\F)\,\,\ll_x\,\, \sqrt{\frac{\delta}{N}}
$$
for $N\leq {1}/{\delta^2}$ and
$$
E_\delta(x,\F)\,\,\ll_x\,\, \delta^{3/2}
$$
for $N> {1}/{\delta^2}$.
\end{theorem}

The above theorem shows that $\liminf_{\#\F\rightarrow \infty}
E_\delta(x,\F) \ll_x \delta^{3/2}$ for unit-norm tight frames obtained through frame paths.
The question is whether $O(\delta^{3/2})$ is sharp. Our next result shows that
in $\R^2$ the bound $O(\delta^{3/2})$ is sharp for {\em all} unit-norm tight frames.
As a result PCM can only partially take advantage of the redundancies in PCM quantization.
We prove the result by showing that the average quantization error for any unit-norm tight
frame in $\R^2$ is bounded from below by $O(\delta^{3/2})$. Set
$$
\EE_\delta(r,\F)\,\,:=\,\, \left(\int_0^{2\pi}
|E_\delta(x_\psi,\F)|^2d\psi\right)^{1/2},
$$
where $x_\psi:=r[\cos\psi,\sin\psi]^T$ and $r>0$. Then

\begin{theorem}\label{th:ave}
Set $R:=r/\delta$ and $\varepsilon:=R+1/2-\lfloor R+1/2\rfloor$.
\begin{enumerate}
\item[(i)]
Suppose $\F$ is an unit-norm tight frame  in $\R^2$.
If $\varepsilon=0$, then
\begin{equation}\label{eq:th6}
\EE_\delta(r,\F)\geq \frac{32}{3\pi^{5/2}}\frac{\delta^{3/2}}{\sqrt{r}}.
\end{equation}

\item[(ii)] Suppose $\F$ is an unit-norm tight frame  in $\R^2$. There exists a
$\varepsilon_0>0$ such that for $\varepsilon \in [0,\varepsilon_0]$ and sufficiently
large $R$ we have
$$
\EE_\delta(r,\F)\geq C\frac{\delta^{3/2}}{\sqrt{r}}
$$
for some fixed constant $C>0$.

\item[(iii)] For the harmonic frame $\tilde{\F}=\{e_j\}_{j=1}^N$ with $e_j=[\cos(2j\pi/N),\sin(2j\pi/N)]^T$, 
assume that $r,\delta$ satisfy
\begin{equation}\label{eq:then0}
\int_{-\pi}^\pi \Delta_\delta(r\cos\theta)\cos\theta d\theta=0,
\end{equation}
where $\Delta_\delta(t):=t-Q_\delta(t)$. Then
$$
 \EE_\delta(r,\tilde{\F})=O(\frac{1}{N})
$$
where $N=\#\tilde{\F}$. In particular, $r,\delta$ satisfy (\ref{eq:then0}) if
$$
   R = \frac{r}{\delta} =\frac{\sqrt{8-2\sqrt{16-\pi^2}}}{\pi}.
$$
\end{enumerate}
\end{theorem}

The above theorem  shows that for any $\delta>0$ there exists a $r_0>0$ such that
$\EE_\delta(r_0,\tilde{\F})\gg \delta^{3/2}$,
which implies that the bound $O(\delta^{3/2})$ is sharp. Hence in general PCM quantization
cannot take full advantage of frame redundancy.
However, by (iii) in Theorem \ref{th:ave}, for each $r>0$, $\EE_\delta(r,\tilde{\F})$ tends to $0$
 at the rate $1/N$ if $\int_{-\pi}^\pi \Delta_\delta(r\cos\theta)\cos\theta d\theta=0$.
The result implies that, for each $\delta>0$ there are some $x$ in $\R^2$ that can take advantage 
of frame redundancy. It will be an interesting problem to find
the sufficient and necessary condition for the validity of (\ref{eq:then0}).

An immediate consequence is that the bound $O(\delta^{3/2})$ is sharp for harmonic frames in
$\R^d$, as any harmonic frame in $\R^d$ has a 2-dimensional harmonic frame imbedded in it,
and the lower bound applies to this imbedded 2-dimensional harmonic frame.

Given the limitation of harmonic frames in mitigating PCM quantization errors, one naturally
asks whether the error bound $O(\delta^{3/2})$ can be improved. It turns out that this is
possible if we distribute the frame elements more evenly on the unit sphere
$\MS^{d-1}$. A sequence of finite sets $A_n \subset \MS^{d-1}$ with cardinality
$N_n = \# A_n$ is said to be {\em asymptotically equidistributed} on $\MS^{d-1}$ if
for any piecewise continuous function $f$ on $\MS^{d-1}$ we have
$$
\lim_{n\rightarrow\infty} \frac{1}{N_n}\sum_{v\in A_n}
f(v)\,\,=\,\,\int_{{z}\in \MS^d}f({z})d\nu,
$$
where $f$ are  piecewise continuous functions on $\MS^d$ and
$d\nu$ denotes the normalized Lebesgue measure on $\MS^{d-1}$.
We have

\begin{theorem}\label{th:lowbound}
    Let $\F_n$ be a unit-norm tight frame in $\R^d$. Assume that $\F_n$ are asymptotically
    equidistributed on $\MS^{d-1}$. Then for any $x\in \R^d$ we have
    $$
      \lim_{n\rightarrow\infty}E_\delta(x,\F_n)\ll_{d} \frac{\delta^{(d+1)/2}}{r^{(d-1)/2}},
    $$
where $r=\|x\|$.
\end{theorem}

Asymptomatically equidistributed unit-norm tight frames in $\R^d$ can be obtained via the
spherical t-design \cite{HolPau04} and other methods. We conjecture that the
bound  $O(\delta ^{(d+1)/2})$ is sharp for any unit-norm tight frame in $\R^{d}$.
 If the conclusion holds, it implies that
asymptotically equidistributed unit-norm tight frames in $\R^d$ are optimal
unit-norm tight frame for PCM quantization.

The paper is organized as follows. After introducing some
preliminaries in Section 2,  we give an up bound of
$E_\delta(x,\F)$ under the WNH, which is valid with high
probability,  in Section 3. We present the proof of Theorem
\ref{th:mupbound}   in Section 4. The
proof of Theorem \ref{th:ave} is
given in Section 5. We finally give the proof of Theorem \ref{th:lowbound}
in Section 6.

\section{Preliminaries}

{\bf Hoeffding's inequality \cite{Hoef63}.} Let $X_1,\ldots,
X_N$ be independent random variables. Assume that for $1\leq j\leq
N$, ${\rm Pr}(X_j-\E(X_j)\in [a_j,b_j])=1$. Then for the sum of
variables,
$$
S\,\,=\,\, X_1+\cdots+X_N
$$
we have the inequality
$$
{\rm Pr}(|S-\E(S)|\geq t) \,\, \leq \,\,
2\exp\left(-\frac{2t^2}{\sum_{j=1}^N(b_j-a_j)^2}\right),
$$
which are valid for positive values of $t$.

{\bf Discrepancy and uniform distribution (see also
\cite{Gunt03}). } Let $\{u_j\}_{j=1}^N$ be a set of points in
$[-1/2,1/2)$ identified with the 1-torus $\T$. The {\em
discrepancy} of $\{u_j\}_{j=1}^N$ is defined by

$$
{\rm Disc}(\{u_j\}_{j=1}^N)\,\,:=\,\,{\rm sup}_{I\subset
\T}\left|\frac{\#(\{u_j\}_{j=1}^N\cap I)}{N}-|I|\right|
$$
where the sup is taken over all subarcs $I$ of $\T$.

We also need the following two well-known results:
\begin{theorem} (Koksma's  inequality) For any sequence of points
$u_1,\ldots,u_N$ in $[-1/2,1/2)$ and any function
$f:[-1/2,1/2)\rightarrow \R$ of bounded variation,
\begin{equation}
\left|\frac{1}{N}\sum_{j=1}^N f(u_j)-\int_{-1/2}^{1/2}f(t)dt
\right| \leq {\rm Var} (f)\cdot {\rm Disc}(\{u_j\}_{j=1}^N),
\end{equation}
where ${\rm Var}(f)$ is the total variation of $f$.
\end{theorem}

\begin{theorem}(Erd\"{o}s-Tur\'{a}n inequality ) For any sequence
of points $u_1,\ldots,u_N$ in $[-1/2,1/2)$, and any positive
integer $K$,
$$
{\rm Disc}(\{u_n\}_{n=1}^N)\,\,\ll \,\,
\frac{1}{K}+\sum_{k=1}^K\frac{1}{k}\left|\frac{1}{N}\sum_{j=1}^Ne^{2\pi
i k u_j }\right|.
$$
\end{theorem}
{\bf Exponential sums.} By Erd\"{o}s-Tur\'{a}n inequality, to
estimate the discrepancy, we need to compute the exponential sums
$$
S=\sum_{m=1}^n e^{2\pi i f(m)},
$$
where $f$ is a real-valued function. We shall use the {\em
truncated Poisson formula} and {\em van der Corput's Lemma} to
estimate $S$.
\begin{theorem}(Truncated Poisson formula) Let $f$ be a real-valued
function and suppose that $f'$ is continuous and increasing on
$[a,b]$. Put $\alpha=f'(a)$ and $\beta=f'(b)$. Then
$$
\sum_{a\leq m\leq b}e^{2\pi i f(m)} =\sum_{\alpha-1 \leq v\leq
\beta+1}\int_a^b e^{2\pi
i(f(\tau)-v\tau)}d\tau+O(\log(2+\beta-\alpha)).
$$
\end{theorem}
\begin{lemma}(van der Corput) Suppose $\phi$ is real-valued and
smooth in the interval $(a,b)$ and that $|\phi^{(r)}(t)|\geq \mu$
for all $t\in (a,b)$ and for a positive integer $r$. If $r=1$,
suppose additionally that $\phi'$ is monotonic. Then
$$
\left|\int_a^b e^{i\phi(t)}dt \right| \leq C_r \mu ^{-1/r},
$$
where $C_r$ is a constant depending on $r$.
\end{lemma}

{\bf Euler-Maclaurin formula.} Suppose $\phi$ is smooth in the interval $[a,b]$, where
$a$ and $b$ are integers. Then
\begin{eqnarray*}
\sum_{j=a}^{b} \phi(j)=\int_{a}^{b}
\phi(x)dx+{(\phi(a)+\phi(b))}/{2}
+\sum_{j=2}^p
(B_{j}/j!)\left(\phi^{(j-1)}(a)-\phi^{(j-1)}(b)\right)+E_p,
\end{eqnarray*}
where  $B_{j}$ are the Bernoulli numbers and
$$
|E_p|\,\,\leq\,\,
\frac{2}{(2\pi)^p}\int_{a}^{b}|\phi^{(p)}(x)|dx.
$$

\section{The error bound under the WNH}

In this section, given $x\in \R^d$, we derive a bound for
$E_\delta(x,\F)$, which is valid with high probability, under
the WNH. As a conclusion, $E_\delta(x,\F)$ tends to $0$ with probability $1$
when $\#\F\rightarrow \infty$ under the WNH.  Recall that ${\mathcal
F}=\{e_j\}_{j=1}^N$ is a finite tight frame in $\R^d$ and  $F=[e_1,\ldots,e_N]$ be the corresponding
matrix whose columns are $\{e_j\}_{j=1}^N$. We define the variation of
 $F$  as
$$
\sigma(F):=\min_p\sum_{j=1}^{N-1}\|e_{p(j)}-e_{p(j+1)}\|.
$$
Then we have
\begin{theorem}\label{th:withwnh}
Under the WNH, for each fixed $x\in \R^d$ and $\varepsilon\in
(0,1/2)$, we have
$$
{\rm Pr}\left(\|x-\x_N\| \leq
\frac{d\delta}{N^{1/2-\varepsilon}}(\sigma(F)+1)\right)\geq
1-2N\exp(-2N^{2\varepsilon}).
$$
\end{theorem}
\begin{proof}
The WNH implies that $x_j-q_j\in [-\delta/2,\delta/2)$ and
$\E(x_j-q_j)=0$. To this end, we set
$$u_j:=\sum_{k=1}^j (x_k-q_k), \quad\quad u_0:=0.$$
Then, by Hoeffding's inequality, we have
$$
{\rm Pr}(|u_j|\leq N^{1/2+\varepsilon}\delta)\,\, \geq \,\,
1-2\exp(-2N^{1+2\varepsilon}/j)\,\, \geq \,\,
1-2\exp(-2N^{2\varepsilon}),
$$
for $j=1,\ldots,N$. We obtain that
$$
{\rm Pr}\left(\bigcap_{j=1}^N(|u_j|\leq
N^{1/2+\varepsilon}\delta)\right)\,\, \geq \,\,
(1-2\exp(-2N^{2\varepsilon}))^N\,\, \geq \,\,
1-2N\exp(-2N^{2\varepsilon}).
$$
Noting that
\begin{eqnarray*}
E_\delta(x,\F) \,\,&=&\,\, \frac{d}{N}\sum_{j=1}^N (x_j-q_j) e_j
\,\,=\,\, \frac{d}{N}\sum_{j=1}^N (u_j-u_{j-1})e_j \\
\,\,&=&\,\, \frac{d}{N}\left(\sum_{j=1}^N u_je_j-\sum_{j=1}^{N-1}
u_je_{j+1}\right)\\
\,\,&=&\,\, \frac{d}{N}\left(\sum_{j=1}^{N-1}
u_j(e_j-e_{j+1})+u_Ne_N\right),
\end{eqnarray*}
we have
\begin{eqnarray*}
E_\delta(x,\F) \leq \frac{d\cdot
N^{1/2+\varepsilon}\delta}{N}(\sigma(F)+1)=\frac{d\cdot
\delta}{N^{1/2-\varepsilon}}(\sigma(F)+1),
\end{eqnarray*}
with probability $1-2N\exp(-2N^{2\varepsilon})$.
\end{proof}

\section{The Proof of Theorem \ref{th:mupbound}}

Let $f:[0,1]\rightarrow \R^d$ be a uniform frame path and set $\F:=\{e_j\}_{j=1}^N$ with $e_j=f(\frac{j-1}{N})$.
 For $x\in \R^d$, we use $\{c_j\}_{j=1}^N$ to denote the corresponding
sequence of frame coefficients with respect to $\F$, i.e.
$c_j=\left<x,f(\frac{j-1}{N})\right>$. Let $\{q_j\}_{j=1}^N$ be the
PCM quantize, i.e.
$q_j=Q_\delta(c_j)$. The resulting quantized expansion is
$$
\x_\F\,\,=\,\,\frac{d}{N}\sum_{j=1}^Nq_je_j^N.
$$
We set
$$
u_j:=\sum_{k=1}^j(c_k-q_k),\,\, j=1,\ldots, N, \text{ and }
u_0:=0.
$$
Then we have
\begin{eqnarray*}
x-\x_\F\,\,&=&\,\,\frac{d}{N}\sum_{j=1}^N (c_j-q_j)e_j
= \frac{d}{N}\sum_{j=1}^N (u_j-u_{j-1})e_j\nonumber\\
&=& \frac{d}{N}\left(\sum_{j=1}^N u_je_j-\sum_{j=1}^{N-1}
u_je_{j+1}\right)\nonumber\\
&=& \frac{d}{N}\left(\sum_{j=1}^{N-1}
u_j(e_j-e_{j+1})+u_Ne_N\right),
\end{eqnarray*}
which implies that
\begin{equation}\label{eq:exp}
\|x-\x_\F\|\,\, =\,\,\frac{d}{N}\|\sum_{j=1}^{N-1}
u_j(e_j-e_{j+1})+u_Ne_N\|.
\end{equation}
 Hence, when working with the approximation error written as
(\ref{eq:exp}), the main step is to find a good estimate for
$u_j$.

\begin{lemma}\label{le:un}
 Suppose that there exists an analytic function $h$ such that $c_j =h({(j-1)}/{N})$
 where $N\geq d$ and $1 \leq j \leq N$.
Suppose further that the zeros of $h''(t)$ on $[0,1]$ are simple. Then
$$
\max_{1\leq j\leq N}|u_j|\,\,\ll_h\,\, \sqrt{N}\log
N\delta+\sqrt{N\delta}+N\delta^{3/2}.
$$
\end{lemma}
\begin{proof}
  Set $y_n:=c_n-q_n$ and
$\tilde{y}_n:=y_n/\delta= (c_n-q_n)/\delta$ where$1\leq n\leq N$. Recall that
we use ${\rm Disc}(\cdot)$ to denote the discrepancy of a
sequence.  Koksma's inequality implies that
\begin{eqnarray*}
|u_j|&=&\delta\left|\sum_{n=1}^j\tilde{y}_n\right|=j\delta
\left|\frac{1}{j}\sum_{n=1}^j\tilde{y}_n-\int_{-1/2}^{1/2}y dy\right|\\
&\leq& j\,\delta\, {\rm Disc }(\{\tilde{y}_n\}_{n=1}^j).
\end{eqnarray*}
Using Erd\"{o}s-T\'{u}ran inequality, one has
$$
\mbox{for any }\, K\, \in \N,\,\, {\rm Disc
}(\{\tilde{y}_n\}_{n=1}^j)\,\,\leq\,\,
\frac{1}{K}+\frac{1}{j}\sum_{k=1}^K\frac{1}{k}\left|\sum_{n=1}^j
e^{2\pi i k \tilde{y}_n} \right|.
$$
Now we need to estimate
$$
\left|\sum_{n=1}^j e^{2\pi i k \tilde{y}_n} \right|.
$$
Set
\begin{equation}\label{eq:xh}
X_N(\cdot)\,\,:=\,\,h(\cdot/N).
\end{equation}
Then we have
$$
y_n=X_N(n)  \mbox{ modulo } [-\delta/2,\, \delta/2).
$$
Since $h$ is analytic, the number of zeros of $h''$ on $[0,1]$ is finite.
Let $\{z_t\}_{t=1}^{n^*}$ be the set of zeros of $h''$ on $[0,1]$,
and let $0<\alpha <1$ be a fixed constant to be specified later.
Without loss of generality, we suppose
$z_t<z_{t+1},t=1,\ldots,n^*-1$. Define the intervals $I_t$ and
$J_t$ by
\begin{eqnarray*}
\mbox{for  } t=1,\ldots,n^*,\,\,\,\,\, I_t&=&[Nz_t-N^{\alpha},\,\,
Nz_t+N^{\alpha}], \\
\mbox{for  } t=1,\ldots,n^*-1,\,\,\,\,\, J_t&=&[Nz_t+N^{\alpha},\,\,
Nz_{t+1}-N^{\alpha}],
\end{eqnarray*}
and
$$
J_0=[1,Nz_1-N^\alpha] \text{ and }  J_{n^*}=[Nz_{n^*}+N^\alpha, N].
$$
If $z_1=0$ , we modify $I_1$ as $[1, N^\alpha]$ and no longer need
$J_0$. Similarly, if $z_{n^*}=1$, we change  $I_{n^*}$ as
$[N-N^{\alpha},N]$ and remove $J_{n^*}$.
Note that
$$
[1,N]\,\,\subset\,\, J_0\cup I_1\cup J_1\cup \cdots \cup I_{n^*}\cup
J_{n^*}.
$$
Since the zeros of $h''$ on $[0,1]$ is simple,  $h'''(z_t)\neq 0$ for $t=1,\ldots,n^*$. 
Then, by Taylor expansion, we have
$$
\mbox{ for } n\in \N \cap J_t, \quad\quad \frac{1}{N^{1-\alpha }}
\,\,=\,\, \frac{N^\alpha}{N}\,\, \ll_h \,\, \left|
h''(\frac{n}{N})\right|
$$
provided $N$ is large enough,  which implies that
$$
\mbox{ for } n\in \N\cap J_t,\,\,\,\,\,\,\, \frac{ k}{
N^{3-\alpha}\cdot \delta}\,\,\ll_h\,\, \Bigl|\frac{k}{\delta}X_N''(n)\Bigr|
$$
for large enough $N$.
Since $h'$ is bounded on $[0,1]$, by (\ref{eq:xh}), we
have
$$
\mbox{ for } n\in \N\cap J_t,\,\,\,\,\,\,\,
\Bigl|\frac{k}{\delta}X_N'(n)\Bigr|\ll_h \frac{k}{N\cdot \delta}.
$$
 Note than $X_N'$ is a monotonic function in $J_t$ and set
$$
\alpha_t:=\min_{n\in \N\cap J_t}\frac{k}{\delta}X_N'(n),\quad\quad \beta_t:=\max_{n\in \N\cap J_t}\frac{k}{\delta}X_N'(n).
$$
Then, a simple observation is that  $\beta_t-\alpha_t\ll_h
\frac{2k}{N\cdot \delta}$.

 Using the {\em truncated Poisson formula} and {\em van der
Corput's Lemma}, we obtain that
\begin{eqnarray*}
& &\left|\sum_{n\in \N\cap J_t} e^{2\pi i k \tilde{y}_n}
\right|\,\,=\,\,\left|\sum_{n\in \N\cap J_t} e^{2\pi i {k
X_N(n)}/{\delta}} \right|\\
& \leq & \sum_{\alpha_t-1\leq v\leq \beta_t+1} \left|
\int_{J_t}e^{2\pi i
(\frac{kX_N(\tau)}{\delta}-v\tau)}d\tau\right|+O(\log(2+\beta_t-\alpha_t))
\\
&\ll_h& (\frac{2k}{N\delta}+2) \left| \int_{J_t}e^{2\pi i
(\frac{kX_N(\tau)}{\delta}-v\tau)}d\tau\right|+O(\log(2+\frac{2k}{N\delta}))\\
&\ll_h & \sqrt{\frac{k}{\delta}}N^{{(1-\alpha)}/{2}}+
\sqrt{\frac{\delta}{k}}N^{(3-\alpha)/2}+O(\log(2+\frac{2k}{N\delta})).
\end{eqnarray*}
The estimate above is also valid if we restrict $n\in [1, j]$, i.e.,
$$
\left|\sum_{n\in \N\cap J_t\cap [1,j]} e^{2\pi i k \tilde{y}_n}
\right| \ll_h  \sqrt{\frac{k}{\delta}}N^{{(1-\alpha)}/{2}}+
\sqrt{\frac{\delta}{k}}N^{(3-\alpha)/2}+O(\log(2+\frac{2k}{N\delta})).
$$
 We also have the trivial estimate:
$$
 \left|\sum_{n\in \N\cap I_t} e^{2\pi
i k \tilde{y}_n} \right|\,\,\leq\,\, 2N^{\alpha}.
$$
Hence, we have
$$
\left| \sum_{n=1}^j e^{2\pi i k \tilde{y}_n}\right|\,\,\ll_h\,\,
2N^{\alpha}+\sqrt{\frac{k}{\delta}}N^{{(1-\alpha)}/{2}}+
\sqrt{\frac{\delta}{k}}N^{(3-\alpha)/2}+O(\log(2+\frac{2k}{N\delta})).
$$
Now, we can estimate ${\rm Disc }(\{\tilde{y}_n\}_{n=1}^j)$ as
follows:
\begin{eqnarray*}
& &\mbox{ for any } K\in \N,\,\, {\rm Disc
}(\{\tilde{y}_n\}_{n=1}^j)\,\,\ll\,\,
\frac{1}{K}+\frac{1}{j}\sum_{k=1}^K\frac{1}{k}\left|\sum_{n=1}^j
e^{2\pi i k \tilde{y}_n} \right|\\
&\ll_h & \frac{1}{K}+\frac{1}{j}\sum_{k=1}^K\frac{1}{k}
\left(2N^{\alpha}+2\sqrt{\frac{k}{\delta}}N^{\frac{1-\alpha}{2}}+
\sqrt{\frac{\delta}{k}}N^{\frac{3-\alpha}{2}}+O(\log(2+\frac{2k}{N\delta}))\right)\\
&\ll_h & \frac{1}{K}+\frac{1}{j}\left(2N^{\alpha}\log K
+2\sqrt{\frac{K}{\delta}}N^{\frac{1-\alpha}{2}}+\sqrt{\frac{\delta}{K}}N^{\frac{3-\alpha}{2}}+O(\sum_{k=1}^K
\frac{1}{k}\log(2+\frac{2k}{N\delta}))\right).
\end{eqnarray*}
We choose  $K=\lfloor\sqrt{N}\rfloor$ and $\alpha=1/2$. Then
\begin{eqnarray*}
|u_j|&\leq&  j\,\delta\, {\rm Disc }(\{\tilde{y}_n\}_{n=1}^j \\
&\ll_h& \frac{j\delta}{\sqrt{N}}+\left( \sqrt{N}\log N \delta
+{2\sqrt{N\delta}}+N{\delta^{3/2}}+ O(\delta \log
N\log(2+\frac{2}{\sqrt{N}\delta})) \right)\\
&\ll & \frac{j\delta}{\sqrt{N}}+\left(\sqrt{N}\log N \delta
+{2\sqrt{N\delta}}+N{\delta^{3/2}}\right),
\end{eqnarray*}
which follows
$$ \max_{1\leq j\leq N}|u_j|\,\,\ll_h\,\,
\sqrt{N}\log N\delta+\sqrt{N\delta}+N\delta^{3/2}.
$$
\end{proof}

\begin{proof}[Proof of Theorem \ref{th:mupbound}]
 Set $y_j:=x_j-q_j$. We
consider
\begin{eqnarray}
\|x-\tilde{x}_N\|\,\,&=&\,\, \|\frac{d}{N}\sum_{j=1}^Ny_j e_j\|\nonumber\\
&=& \frac{d}{N}\|\sum_{j=1}^{N-1}u_j(e_j-e_{j+1})+u_Ne_N\|  \ll
\frac{d}{N}\max_{1\leq j\leq N} |u_j|,\label{eq:xuj}
\end{eqnarray}
where the last inequality follows by $\|e_j-e_{j+1}\|=\|f(\frac{j-1}{N})-f(\frac{j}{N})\| \ll \frac{1}{N}$
with $\|f'\|$ being bounded. Lemma \ref{le:un} with $h(t) = \left<x,f(t)\right>$
implies that
\begin{equation}\label{eq:un3}
\frac{1}{N}\max_{1\leq j\leq N}|u_j|\,\, \ll_x\,\,
\sqrt{\frac{\delta}{N}}+\frac{(\log N)\cdot
\delta}{\sqrt{N}}+\delta^{3/2}.
\end{equation}
A simple observation is that
\begin{equation}\label{eq:un31}
\max\Bigl\{\sqrt{\frac{\delta}{N}}, \frac{\delta \log N}{\sqrt{N}},
\delta^{3/2}\Bigr\}\,\, \leq\,\, \sqrt{\frac{\delta}{N}}
\end{equation}
when  $N \leq \frac{1}{\delta^2}$. Combing (\ref{eq:un3}) and
(\ref{eq:un31}), we have
$$
\frac{1}{N}\max_{1\leq j\leq N} |u_j|\,\, \ll_x \,\,
\sqrt{\frac{\delta}{N}}, \,\, {\rm when }\,\,  N\leq
\frac{1}{\delta^2},
$$
which implies that
\begin{eqnarray*}
\|x-\tilde{x}_\F\|\,\, \ll_x \,\, \sqrt{\frac{\delta}{N}}
\end{eqnarray*}
provided $N\leq \frac{1}{\delta^2}$.

We now turn to the case where $N\geq \frac{1}{\delta^2}$. To this end, in the basis of (\ref{eq:un3}),
  we only need to prove that
$$
\max\Bigl\{\sqrt{\frac{\delta}{N}}, \frac{\delta \log N}{\sqrt{N}},
\delta^{3/2}\Bigr\}\,\, \leq\,\, \delta^{3/2}.
$$
Indeed, $N\geq {1}/{\delta^2}$ implies that $
\sqrt{\frac{\delta}{N}}\,\, \leq\,\, \delta^{3/2}$. Moreover, we
have
$$\frac{\delta \log N}{\sqrt{N}}\,=\,\frac{\delta}{N^{1/4}}
\frac{\log N}{N^{1/4}}\,\leq\,  \frac{\delta}{N^{1/4}}\, \leq\,
\delta^{3/2}
$$
 provided $N$ is big enough. The claim follows.
\end{proof}

\section{The proof of  Theorem \ref{th:ave}}
Throughout the rest of this paper, we set
$$
\Delta_\delta(t)\,\,:=\,\, t-Q_\delta(t).
$$
Then we have
\begin{lemma}\label{le:low1}
Suppose that $r>0, \delta>0$. Set  $R:=r/\delta$ and $\varepsilon:=R+1/2-\lfloor R+1/2\rfloor$.
\begin{enumerate}
\item[(i)] If $\varepsilon=0$, then
\begin{equation*}
\frac{16\sqrt{2}}{3\pi^2}\frac{\delta^{3/2}}{\sqrt{r}}\,\, \leq\,\,
\int_{-\pi}^\pi \Delta_\delta(r\cos\theta)\cos\theta
d\theta.
\end{equation*}
\item[(ii)]  There exists a
$\varepsilon_0>0$ such that, when $\varepsilon \in [0,\varepsilon_0]$,
\begin{equation*}
C_1\frac{\delta^{3/2}}{\sqrt{r}}\,\, \leq\,\,
\int_{-\pi}^\pi \Delta_\delta(r\cos\theta)\cos\theta
d\theta,
\end{equation*}
provided $R$ is big enough, where $C_1$ is a fixed constant.
\end{enumerate}

\end{lemma}

\begin{proof}
We first consider the case with $\varepsilon=0$.
Note that
\begin{eqnarray}
& &\int_{-\pi}^\pi \Delta_\delta(r\cos\theta)\cos\theta d\theta
 =2\int_{0}^\pi \Delta_\delta(r\cos\theta)\cos\theta d\theta\nonumber\\
 &=& 2\int_0^\pi (r\cos\theta -\delta
 \lfloor{R\cos\theta}+{1}/{2}\rfloor)\cos\theta
 d\theta\nonumber\\
 &=& \pi r-2\delta \int_0^\pi \lfloor R\cos\theta
 +{1}/{2}\rfloor\cos\theta d\theta \nonumber\\
 &=&\delta \left(\pi R-2\int_0^\pi \lfloor  R\cos\theta
 +{1}/{2}\rfloor\cos\theta d\theta\right).\label{eq:eq1}
\end{eqnarray}
If $R=1/2$, then a simple calculation shows that
$$
\int_{-\pi}^\pi \Delta_\delta(r\cos\theta)\cos\theta d\theta=\pi \delta/2
$$
which implies the conclusion.
To this end, we only need investigate the case $R\in \Z_{\geq 1}+1/2$.
We set $ L:=-R+{1}/{2}\in \Z$ and $U:=R+{1}/{2}\in \Z$.
We now consider
\begin{eqnarray}
 & & \int_0^\pi \lfloor R\cos\theta
 +{1}/{2}\rfloor\cos\theta d\theta\nonumber\\
 &=&
\sum_{j=L}^{U-1}j\int_{\arccos((j+1/2)/R)}^{\arccos((j-1/2)/R)}\cos\theta
d\theta\nonumber\\
&=&
\sum_{j=L}^{U-1}j(\sqrt{1-((j-1/2)/R)^2}-\sqrt{1-((j+1/2)/R)^2})\nonumber\\
 &=& \sum_{j=L-1}^{U-2}(j+1)\sqrt{1-((j+1/2)/R)^2}-\sum_{j=L}^{U-1}j\sqrt{1-((j+1/2)/R)^2}\nonumber\\
 &=&
\sum_{j=L}^{U-2}\sqrt{1-((j+1/2)/R)^2}=(\sum_{j=L}^{U-2}\sqrt{R^2-(j+1/2)^2})/R\nonumber\\
&=&\sum_{j=0}^{R-3/2}\left(2\sqrt{R^2-(j+1/2)^2}\right)/R.\label{eq:deng}
\end{eqnarray}
 Set
$\psi_j:=\arccos((j+1/2)/R)-\arccos((j+3/2)/R)$. We  claim that
\begin{equation}\label{eq:claim1}
\int_0^\pi \lfloor R\cos\theta
 +\frac{1}{2}\rfloor\cos\theta
 d\theta=R\sum_{j=0}^{R-3/2}\sin\psi_j+\frac{1}{2}\sqrt{1-\frac{1}{4R^2}}.
\end{equation}
 Indeed, the claim can follow from the following calculation
\begin{eqnarray*}
& &R\sum_{j=0}^{R-3/2}\sin\psi_j\\
&=&\frac{1}{R}\sum_{j=0}^{R-3/2}\left((j+\frac{3}{2})\sqrt{R^2-(j+\frac{1}{2})^2}-(j+\frac{1}{2})\sqrt{R^2-(j+\frac{3}{2})^2}\right)\\
&=&\frac{1}{R}\left(\frac{3}{2}\sqrt{R^2-\frac{1}{4}}+\sum_{j=0}^{R-5/2}2\sqrt{R^2-(j+\frac{3}{2})^2}\right)\\
&=&\int_0^\pi \lfloor R\cos\theta
 +\frac{1}{2}\rfloor\cos\theta
 d\theta-\frac{1}{2}\sqrt{1-\frac{1}{4R^2}},
\end{eqnarray*}
where the last equation follows from (\ref{eq:deng}). Moreover, by Taylor expansion,  we
have
\begin{eqnarray*}
\sum_{j=0}^{R-3/2}\psi_j= \arccos(\frac{1}{2R})\leq \frac{\pi}{2}-\frac{1}{2R},
\end{eqnarray*}
which implies that
\begin{equation}\label{eq:pi}
  2R \sum_{j=0}^{R-3/2}\psi_j+1\leq \pi R.
\end{equation}
Then, combining (\ref{eq:claim1}) and (\ref{eq:pi}), we obtain that
\begin{eqnarray}
& &2R\sum_{j=0}^{R-3/2}(\psi_j-\sin\psi_j) \leq \pi R-2 \int_0^\pi \lfloor
R\cos\theta+\frac{1}{2}\rfloor\cos\theta d\theta \label{eq:lian}
\end{eqnarray}
Noting that $\psi_j-\sin\psi_j>0$, we have
\begin{eqnarray}
& &\sum_{j=0}^{R-3/2}(\psi_j-\sin\psi_j) \geq
\psi_{R-3/2}-\sin\psi_{R-3/2}\label{eq:psi}\\
&\geq &\frac{4}{3\pi^2}\psi_{R-3/2}^3=\frac{4}{3\pi^2}(\arccos(1-1/R))^3 \nonumber\\
 &\geq &\frac{8\sqrt{2}}{3\pi^2}R^{-3/2}.\nonumber
\end{eqnarray}
 Combining (\ref{eq:eq1}),
(\ref{eq:lian}) and (\ref{eq:psi}), we arrive at
\begin{eqnarray*}
& &\int_{-\pi}^\pi \Delta_\delta(r\cos\theta)\cos\theta
d\theta\\
 &=&\delta \left(\pi R-2\int_0^\pi \lfloor R\cos\theta
 +\frac{1}{2}\rfloor\cos\theta d\theta\right)\\
 &\geq &\frac{16\sqrt{2}}{3\pi^2}\frac{\delta^{3/2}}{\sqrt{r}}.
\end{eqnarray*}

We next consider the case $\varepsilon\neq 0$. Using the similar method  as before,
we have
$$
\int_0^\pi \lfloor R\cos\theta+\frac{1}{2}\rfloor \cos\theta d\theta
= R \sum_{j=0}^{R-1/2-\varepsilon} \sin\psi_{j-1}-\frac{1}{2}\sqrt{1-\frac{1}{4R^2}}+
(R+1-\varepsilon)\sqrt{\frac{2\varepsilon}{R}-(\frac{\varepsilon}{R})^2}
$$
and
$$
\pi R = 2R\sum_{j=0}^{R-1/2-\varepsilon}\psi_{j-1}
+2\sqrt{2\varepsilon R} -1+\frac{\sqrt{2}}{6}\frac{\varepsilon^{3/2}}{R^{1/2}}+
O(\frac{1}{R^{3/2}}).
$$
Then
\begin{eqnarray}
 & &\pi R-2 \int_0^\pi \lfloor R\cos\theta+1/2\rfloor \cos\theta d\theta\nonumber\\
&=&2 R\sum_{j=0}^{R-1/2-\varepsilon}(\psi_{j-1}-\sin\psi_{j-1})-2{(R+1-\varepsilon)}\sqrt{\frac{2\varepsilon}{R}-(\frac{\varepsilon}{R})^2}\nonumber\\
& &+2\sqrt{2\varepsilon R}
+\frac{\sqrt{2}}{6}\frac{\varepsilon^{3/2}}{\sqrt{R}}+O(\frac{1}{R^{3/2}}) \nonumber\\
&\geq& 2R(\psi_{R-3/2-\varepsilon}-\sin\psi_{R-3/2-\varepsilon})+(\frac{2\sqrt{2}}{3}\varepsilon^{3/2}-2(1-\varepsilon)\sqrt{2\varepsilon})\frac{1}{\sqrt{R}}+
O(\frac{1}{R^{3/2}})\nonumber\\
&=& \left(\frac{(\sqrt{2+2\varepsilon}-\sqrt{2\varepsilon})^3}{3}+\frac{2\sqrt{2}}{3}\varepsilon^{3/2}-2(1-\varepsilon)\sqrt{2\varepsilon}\right)\frac{1}{\sqrt{R}}+O(\frac{1}{R^{3/2}}).
\label{eq:nonint}
\end{eqnarray}
Note that there exists a $\varepsilon_0>0$ such that
$$
\frac{(\sqrt{2+2\varepsilon}-\sqrt{2\varepsilon})^3}{3}+\frac{2\sqrt{2}}{3}\varepsilon^{3/2}-2(1-\varepsilon)\sqrt{2\varepsilon}
$$
is positive when $\varepsilon \in [0,\varepsilon_0]$, which implies that
$$
\pi R-2 \int_0^\pi \lfloor R\cos\theta+1/2\rfloor \cos\theta d\theta \geq \frac{C_1}{\sqrt{R}}
$$
when $R$ is big enough, where $C_1$ is a positive constant.  The conclusion follows.
\end{proof}

\begin{proof}[Proof of Theorem \ref{th:ave}]
We suppose $x_{\psi_0}=r[\cos\psi_0,\sin\psi_0]^T$ and
$\F=\{e_j\}_{j=1}^N$ with $e_j=[\cos\theta_j,\sin\theta_j]^T,
\theta_j\in [0,2\pi)$. Set $R:=r/\delta$. Then
\begin{eqnarray*}
E_\delta(x_{\psi_0},\F)&=&\frac{2}{N}\|\sum_{j=1}^N
\Delta_\delta(r\cos(\theta_j-\psi_0))e_j\| \\
&=& \frac{2\delta}{N}\|\sum_{j=1}^N
\Delta_1(R\cos(\theta_j-\psi_0))e_j\|.
\end{eqnarray*}
Let
$$
P_{\psi_0}:= \left[
  \begin{array}{cc}
    \cos\psi_0 &\sin\psi_0 \\
    -\sin\psi_0& \cos\psi_0\\
  \end{array}
\right]
$$
Then a simple observation is that
\begin{equation}\label{eq:th6E}
E_\delta(x_{\psi_0},\F)= \frac{2\delta}{N}\|\sum_{j=1}^N
\Delta_1(R\cos(\theta_j-\psi_0))P_{\psi_0}e_j\|.
\end{equation}
Denote $\mu_{\F}=\frac{1}{N}\sum_{j=1}^N\delta_{\theta_j}$ where
$\delta_{\theta_j}$ is the Dirac measure at $\theta_j$. Then we
can rewrite (\ref{eq:th6E}) as
$$
E_\delta(x_{\psi_0},\F)= 2\delta \|(H_R *
\mu_{\F})(\psi_0)\|,
$$
where $*$ is the component-wise convolution operator and
$$ H_R(t):=\Delta_1(R\cos t)\left[\begin{array}{c}\cos t \\
-\sin t
\end{array}\right].
$$
Note that
\begin{eqnarray*}
\EE_\delta(r,{\F})&=&\left(\int_0^{2\pi} (E(x_{\psi},{\F}))^2
d\psi\right)^{1/2}\\
&=& 2\delta \|H_R*\mu_{\F}\|_{L^2} =\frac{\sqrt{2}\delta}{\sqrt{\pi}} (\sum_{k\in \Z}\|
\widehat{H_R*\mu_{\F}}(k)\|^2)^{1/2} \\
&=&\frac{\sqrt{2}\delta}{\sqrt{\pi}} (\sum_{k\in \Z}\|\hat{H}_R(k)\|^2
|\hat{\mu}_{\F}(k)|^2)^{1/2},
\end{eqnarray*}
where $\|\cdot\|_{L^2}$  and $\|\cdot\|$ denote the $L^2$ norm of vector functions and $\ell^2$ norm of $\R^2$, respectively.
Since $\hat{\mu}_{\F}(0)=1$, it follows that
\begin{eqnarray*}
\EE_\delta(r,{\F})&=&\frac{\sqrt{2}\delta}{\sqrt{\pi}} (\|\hat{H}_R(0)\|^2+\sum_{k\in \Z\setminus
\{0\}} \|\hat{H}_R(k)\|^2 |\hat{\mu}_{\F}(k)|^2)^{1/2}\\
&\geq & \frac{\sqrt{2}\delta}{\sqrt{\pi}}  \|\hat{H}_R(0)\|.
\end{eqnarray*}
We still need to estimate $\hat{H}_R(0)$. Note
\begin{eqnarray*}
\hat{H}_R(0)&=&\int_0^{2\pi}H_R(t)dt=\int_0^{2\pi} \Delta_1(R\cos
t)\left[\begin{array}{c}\cos t \\ -\sin t\end{array}\right] dt \\
&=&\left[\begin{array}{c} \int_0^{2\pi }\Delta_1(R\cos t)\cos t dt \\
0
\end{array}\right].
\end{eqnarray*}
By (i) in Lemma \ref{le:low1}, when $\delta$ is small enough,
$$
\int_0^{2\pi }\Delta_1(R\cos t)\cos t dt =
\frac{1}{\delta}\int_0^{2\pi }\Delta_\delta(r\cos t)\cos t dt
\geq \frac{16\sqrt{2}}{3\pi^2}\sqrt{\frac{\delta}{{r}}},
$$
which implies that
$$
\EE_\delta(r,{\F}) \geq \frac{32}{3\pi^{5/2}}\frac{\delta^{3/2}}{\sqrt{r}}.
$$
Similarly, (ii) can be proved by (ii) in Lemma \ref{le:low1}.

We now turn to (iii). Note that $\hat{H}_R(0)=0$ if
\begin{equation}\label{eq:deng0}
\int_{-\pi}^\pi \Delta_\delta(r\cos\theta)\cos\theta d\theta=0.
\end{equation}
Also, note that
$$
\hat{\mu}_{\tilde{\F}}(k)=\begin{cases}
1, &  k\, { \rm mod }\, N=0,\\
0,  & k\, { \rm mod }\, N\neq 0.
\end{cases}
$$
where $N=\#\tilde{\F}$. Hence,
$$
\EE_\delta(r,\tilde{\F})=\frac{\sqrt{2}\delta}{\sqrt{\pi}}\left(\sum_{k\in \Z\setminus\{0\}}\|\hat{H}_R(kN)\|^2\right)^{1/2}.
$$
According to Riemann-Lebesgue Lemma,
$$
 \|\hat{H}_R(N)\|=O(\frac{1}{N}).
$$
The conclusion follows.

We now consider (\ref{eq:deng0}).  A simple calculation shows that
$$
\int_{-\pi}^\pi \Delta_{\delta}(r\cos\theta)\cos\theta d\theta= \pi r-4\delta\sqrt{1-\frac{\delta^2}{4r^2}}
$$
provided $r\leq \delta \leq 2r$. The result implies that
$$
\hat{H}_R(0)=0
$$
if
$$
   R = \frac{r}{\delta} =\frac{\sqrt{8-2\sqrt{16-\pi^2}}}{\pi}.
$$

\end{proof}

\section{The proof of Theorem \ref{th:lowbound}}

To prove Theorem \ref{th:lowbound}, we first introduce a lemma.

\begin{lemma}\label{le:sanjiao}
When $d\geq 2$,
$$
\left| \int_0^\pi
\Delta_\delta(r\cos\theta)\cos\theta(\sin\theta)^{d-2}d\theta\right|\,\,\ll_d
\,\,  \frac{\delta^{(d+1)/2}}{r^{(d-1)/2}}.
$$
\end{lemma}
\begin{proof}
Similar with Lemma \ref{le:low1}, we set
  $R:=r/\delta$. To state conveniently, we only consider the case  $R+1/2\in \Z$. The other case
  can be proved by a similar method.
A simple observation is that
\begin{eqnarray}
& &\int_0^\pi
\Delta_\delta(r\cos\theta)\cos\theta(\sin\theta)^{d-2}d\theta= \nonumber\\
& &\delta  \left(R \int_0^\pi
(\cos\theta)^2(\sin\theta)^{d-2}d\theta-\int_0^\pi \lfloor
R\cos\theta+\frac{1}{2}\rfloor\cos\theta (\sin\theta)^{d-2}d\theta
\right).\label{eq:he2}
\end{eqnarray}
Then, we  consider
\begin{eqnarray}
& &\int_0^\pi \lfloor R\cos\theta+\frac{1}{2}\rfloor\cos\theta
(\sin\theta)^{d-2}d\theta \nonumber\\
&=& \sum_{j=-R+1/2}^{R-1/2}
j\int_{\arccos((j+1/2)/R)}^{\arccos((j-1/2)/R)}\cos\theta
(\sin\theta)^{d-2}d\theta \nonumber\\
&=& \frac{1}{d-1}\sum_{j=-R+1/2}^{R-1/2}
j((1-((j-1/2)/R)^2)^{(d-1)/2}-(1-((j+1/2)/R)^2)^{(d-1)/2})\nonumber\\
&=&
\frac{1}{d-1}\sum_{j=-R+3/2}^{R-1/2}(1-((j-1/2)/R)^2)^{(d-1)/2}.\label{eq:jifen}
\end{eqnarray}
As we shall see later,
\begin{eqnarray}\label{eq:he1}
& &\frac{1}{d-1}
\sum_{j=-R+3/2}^{R-1/2}(1-((j-1/2)/R)^2)^{(d-1)/2}\nonumber\\
&=&R\int_0^\pi(\cos\theta)^{2} (\sin\theta)^{d-2}
d\theta+O({1}/{R^{(d-1)/2}}).
\end{eqnarray}
 Then combining (\ref{eq:he2}), (\ref{eq:jifen}) and
(\ref{eq:he1}), we reach the conclusion.

We remain to argue (\ref{eq:he1}). We set
$$
f(x)\,\,:=\,\,\frac{1}{d-1}
\left(1-\left((2x-1)/(2R)\right)^2\right)^{(d-1)/2}.
$$
Then the left side of (\ref{eq:he1}) equals to
$\sum_{j=-R+3/2}^{R-1/2} f(j)$.
 Recall that
Euler-Maclaurin formula
\begin{eqnarray*}
\sum_{j=-R+3/2}^{R-1/2} f(j)\,\,&=&\,\, \int_{-R+3/2}^{R-1/2}
f(x)dx+{(f(-R+3/2)+f(R-1/2))}/{2}\\
& &+\sum_{j=2}^p
(B_{j}/j!)\left(f^{(j-1)}(-R+3/2)-f^{(j-1)}(R-1/2)\right)+E_p,
\end{eqnarray*}
where $B_{j}$ are the Bernoulli numbers and
$$
|E_p|\,\,\leq\,\,
\frac{2}{(2\pi)^p}\int_{-R+3/2}^{R-1/2}|f^{(p)}(x)|dx.
$$
 Note that
\begin{eqnarray*}
\int_{-R+3/2}^{R-1/2}f(x)dx=\int_{-R+1/2}^{R+1/2}f(x)dx+O(\frac{1}{R^{(d-1)/2}}).
\end{eqnarray*}
Let $\theta=\arccos ((2x-1)/2R)$. We have
$$
\int_{-R+1/2}^{R+1/2}f(x)dx\,\,=\,\, \frac{R}{d-1}\int_0^\pi
(\sin\theta)^{d} d\theta \,\,=\,\, R\int_0^\pi(\cos\theta)^{2}
(\sin\theta)^{d-2} d\theta,
$$
where the last equality follows from the integration by parts.
Then we arrive at
\begin{eqnarray*}
\int_{-R+3/2}^{R-1/2}f(x)dx=R\int_0^\pi(\cos\theta)^{2}
(\sin\theta)^{d-2} d\theta +O(\frac{1}{R^{(d-1)/2}}).
\end{eqnarray*}
 Note that
\begin{equation}\label{eq:comb1}
f(-R+3/2)=f(R-1/2)\ll \frac{1}{R^{(d-1)/2}}
\end{equation}
 and
\begin{equation}\label{eq:comb2}
|f^{(j-1)}(-R+3/2)-f^{(j-1)}(R-1/2)|\ll_j  \frac{1}{R^{(d-1)/2}},
\quad\quad j\geq 2.
\end{equation}
Hence, to prove (\ref{eq:he1}), we just need estimate the error
term in Euler-Maclaurin formula. We first consider the case where
$d-1$ is an even number. We take $p=d$ in Euler-Maclaurin formula,
 and then $E_p=0$ with $f^{(j)}\equiv 0$ provided $j\geq d$. Then
(\ref{eq:he1}) follows when $d-1$ is even.

We turn to the case where $d-1$ is odd.  We consider the $j$th
derivative of $f$. Recall that $\theta(x)=\arccos ((2x-1)/2R)$ is a
function about $x$. Then, using the new variable $\theta$,
$f(x)=(\sin\theta)^{d-1}/(d-1)$ and $\theta'(x)=-1/(R\sin \theta)$. A
simple calculation shows that
\begin{eqnarray*}
f'(x)&=&f'(\theta)\theta'(x)=-\frac{1}{R}(\sin\theta)^{d-3}\cos\theta,\\
f''(x)&=&\frac{1}{R^2}\left((d-3)(\sin\theta)^{d-5}-(d-2)(\sin\theta)^{d-3}\right).
\end{eqnarray*}
Then, by induction, $f^{(2j)}$ is in the form of
$$
\frac{1}{R^{2j}}\sum_{j\leq k\leq 2j}
C_{k,d-1}(\sin\theta)^{d-2k-1},\,\, C_{k,d-1}\in \R,
$$ while $f^{(2j+1)}$ is in the form of
$$
\frac{1}{R^{2j+1}}\sum_{j+1\leq k\leq 2j+1}
C_{k,d-1}'(\sin\theta)^{d-2k-1}\cos\theta,\,\, C_{k,d-1}'\in \R.
$$
We take $p=(d+2)/2$ in Euler-Maclaurin formula.  Then
\begin{equation}\label{eq:comb3}
|E_p| \ll_p
\frac{1}{R^{d/2}}\int_{\arccos(1-1/R)}^{\arccos(-1+1/R)}
\frac{1}{\sin^2\theta}d\theta \ll \frac{1}{R^{(d-1)/2}}.
\end{equation}
when $p$ is even.  Similar argument also holds when $p$ is odd.
Combining (\ref{eq:comb1}),  (\ref{eq:comb2}) and (\ref{eq:comb3}), we arrive at (\ref{eq:he1}). The conclusion follows.
\end{proof}

\begin{proof}[Proof of Theorem \ref{th:lowbound}] We denote the
number of the non-zero entries in $x$ by $\|x\|_0$, i.e.,
$$
\|x\|_0 \,\,:=\,\,\#\{j:x_j\neq 0\}.
$$
 The proof is by induction on $\|x\|_0$. Note that
\begin{eqnarray*}
& &\lim_{n\rightarrow \infty}E_\delta(x,\F_n)
=\lim_{n\rightarrow
\infty}\|\frac{d}{N_n}\sum_{j=1}^{N_n}\Delta_\delta(x\cdot e_j
)e_j \|\\
&=&d\, \|\int_{{\bf z}\in \MS^d}\Delta_\delta(x\cdot {\bf z}
){\bf z}d\omega \|.
\end{eqnarray*}
 We begin with $\|x\|_0=1$.
 Without loss of generality, we suppose
$x=[x_1,0,\ldots,0]^T\in \R^{d}$ and consider
$\lim_{n\rightarrow \infty}E_\delta(x,\F_n)$. By the
sphere coordinate system, each ${\bf z}=[z_1,\ldots,z_{d-1}]\in \MS^{d-1}$
can be written in the form of
\begin{eqnarray*}
[\cos\theta_1,\sin\theta_1\cos\theta_2,
\sin\theta_1\sin\theta_2\cos\theta_3,\ldots,\sin\theta_1\cdots
\sin\theta_{d-1}]^T,
\end{eqnarray*}
where $\theta_1\in [0,\pi)$ and $\theta_j\in [-\pi,\pi), 2\leq
j\leq d-1$. To state conveniently, we set
$$
\Theta\,\,:=\,\, [0,\pi)\times \underbrace{[-\pi, \pi)\times
\cdots \times [-\pi,\pi)}_{d-2}, \qquad {S}_m(\theta):=\prod_{j=1}^{m}\sin\theta_j
$$
and
$$
J_d(\theta):=\left|(\sin\theta_1)^{d-2}(\sin\theta_2)^{d-3}\cdots
(\sin\theta_{d-2})\right|.
$$
 Noting that
$$
d\omega= { J}_d(\theta) d\theta_1\cdots d\theta_{d-1}\quad
\mbox{ and }\quad \int_{{\bf z}\in \MS^{d-1}}\Delta_\delta(x_1  z_1){
z}_jd\omega=0,\quad\quad 2\leq j\leq d-1,
$$
 we have
\begin{eqnarray*}
& &\lim_{n\rightarrow \infty}E_\delta(x,\F_n)
=d\,\|\int_{{\bf z}\in \MS^{d-1}}\Delta_\delta(x\cdot
{\bf z} ){\bf z}d\omega \|\\
&=& d\,\left|\int_{{\bf z}\in \MS^{d-1}}\Delta_\delta(x_1
{z}_1 ){z}_1d\omega \right|\\
&=& d\,\left|\int_{\theta\in
\Theta}\Delta_\delta(x_1\cos\theta_1)\cos\theta_1(\sin\theta_1)^{d-2}
|(\sin\theta_2)^{d-3}\cdots (\sin\theta_{d-2})| d \theta_1\cdots
d\theta_{d-1}\right|\\
&\ll_d& \delta^{(d+1)/2}/|x_1|^{(d-1)/2}
\end{eqnarray*}
where the last inequality  follows from  Lemma \ref{le:sanjiao}.

For the induction step, we suppose that the conclusion holds for
the case where $\|x\|_0\leq k$. We now consider $\|x\|_0 \leq
k+1$. Without loss of generality, we suppose $x$ is in the form of
$[0,\ldots,0,x_{d-k+1},\ldots,x_{d}]\in \R^{d}$. We can write
$[x_{d-1},x_{d}]$ in the form of $(r_0\cos\varphi_0,r_0\sin\varphi_0)$,
where $r_0\in\R_+$ and $\varphi_0\in [0,2\pi)$. Then
$$
x\cdot {\bf z} =\sum_{m=d-k+1}^{d-2}x_mS_m(\theta)\cos\theta_m+
r_0\sin\theta_1\cdots \sin\theta_{d-2}\cos(\theta_{d-1}-\varphi_0)=:{
T}(\varphi_0).
$$
A simple observation is
\begin{eqnarray*}
& &\left(\int_{\theta\in\Theta}\Delta_\delta({ T}(\varphi_0)){
S}_{d-2}(\theta){ J}_d(\theta)\cos\theta_{d-1}
d\theta\right)^2+\left(\int_{\theta\in\Theta}\Delta_\delta({
T}(\varphi_0)){ S}_{d-2}(\theta){ J}_{d}(\theta) \sin\theta_{d-1}
d\theta\right)^2\\
&=&\left(\int_{\theta\in\Theta}\Delta_\delta({ T}(0)){
S}_{d-2}(\theta){ J}_d(\theta)\cos\theta_{d-1}
d\theta\right)^2+\left(\int_{\theta\in\Theta}\Delta_\delta({
T}(0)){ S}_{d-2}(\theta){ J}_d(\theta)\sin\theta_{d-1}
d\theta\right)^2.
\end{eqnarray*}
Then we have
\begin{eqnarray*}
& &\lim_{n\rightarrow \infty}E_\delta(x,\F_n)
=d\|\int_{{\bf z}\in \MS^{d-1}}\Delta_\delta(x\cdot
{\bf z} ){\bf z}d\omega \|\\
&=&d\left(\sum_{m=d-k+1}^{d-1}(\int_{\theta\in\Theta}\Delta_\delta(T(\varphi_0))S_{m-1}(\theta)J_d(\theta)\cos\theta_m
d\theta)^2+(\int_{\theta\in\Theta}\Delta_\delta(T(\varphi_0))S_{d-1}(\theta)J_d(\theta)
d\theta)^2\right)^{1/2}\\
&=&d\left(\sum_{m=d-k+1}^{d-1}(\int_{\theta\in\Theta}\Delta_\delta(T(0))S_{m-1}(\theta)J_d(\theta)\cos\theta_m
d\theta)^2+(\int_{\theta\in\Theta}\Delta_\delta(T(0))S_{d-1}(\theta)J_d(\theta)
d\theta)^2\right)^{1/2}\\
&\ll_d &  \delta^{(d+1)/2}/r^{(d-1)/2}
\end{eqnarray*}
where the last inequality follows from the fact $\|x\|_0\leq k$
provided $\varphi_0=0$.

\end{proof}

\bigskip

\bibliographystyle{amsplain}

\begin{thebibliography}{999}


\bibitem{Benn48} W. Bennett, Spectra of quantized signals, Bell
Syst. Tech. J. 27 (1948), pp.446-472.

\bibitem{BodPau07} B. Bodmann and V. Paulsen,
Frame paths and error bounds for sigma-delta quantization,
Appl. Comput. Harmon. Anal. 22 (2007), pp.176-197.

\bibitem{BorWan09} S. Borodachov and Y. Wang
On the distribution of uniform quantization errors,
Applied and Computational Harmonic Analysis 27 (2009),pp. 334-341.


\bibitem{BePoYi06} J. Benedetto, A. M. Powell and \"{O}. Y{\i}lmaz,
Sigma-delta quantization and finite frames, IEEE Trans. Inform.
Theory, 52(2006),pp. 1990-2005.

\bibitem{BePoYi06II} J. Benedetto, A. M. Powell and \"{O}. Y{\i}lmaz, Second
order sigmal-delta quantization of finite frame expansions, Appl.
Comput. Harmon. Anal., 20 (2006), pp.126-248.


\bibitem{BorWan09} S. Borodachov and Y. Wang, Lattice quantization
error for redundant representations, Appl. Comput. Harmon. Anal.
27 (2009), pp. 334-341.

%\bibitem{cvetkovic}{Cvet03} Z. Cvetkovi\'{c}, Reslilience properties of
%redundant expansions under additive noise quantization, IEEE Trans. Inform.
%Theory 49(3) (2003) 644-656.


\bibitem{GoVeTh98} V. Goyal, M. Vetterli, N. Thao, Quantized
overcomplete expansions in $\R^n$: Analysis, synthesis, and
algorithms, IEEE Trans. Inform.
Theory 44(1998), pp.16-31.


\bibitem{Gunt03} S. G\"{u}nt\"{u}rk, Approximating a bandlimited
function using very coarsely quantized data: improved error
estimates in sigma-delta modulation, J. Amer. Math. Soc., 17(2003), pp.229-242.



\bibitem{Hoef63}W. Hoeffding,
Probability inequalities for sums of bounded random variables, J.
Amer. Stat. Asso. 58 (1963), pp.13-30.

\bibitem{HolPau04} R. Holmes and V. Paulsen, Optimal frames for
erasures, Linear Algebra Appl., 377(2004), pp.31-51.

\bibitem{JiWaWa07} D. Jimenez, L. Wang and Y. Wang, White noise
hypothesis for uniform quantization errors, SIAM J. Math. Anal.,
38(2007), pp.2042-2056.



\end{thebibliography}

\end{document}